\documentclass[leqno,11pt,a4paper]{amsart}
\usepackage[usenames,dvipsnames]{color}
\usepackage{hyperref}
\usepackage{ragged2e}
\usepackage{multirow}
\usepackage{graphicx}
\newtheorem{thm}{Theorem}[section]
\newtheorem{lem}[thm]{Lemma}

\newtheorem{prop}[thm]{Proposition}
\theoremstyle{definition}

\theoremstyle{remark}

\theoremstyle{remark}

\numberwithin{equation}{section}

\renewcommand{\Re}[1]{\mathrm{Re}\!\left\{ #1\right\}}
\renewcommand{\Im}[1]{\mathrm{Im}\!\left\{ #1\right\}}
\begin{document}
\author[M.~Rodr\'iguez]{Miguel Rodr\'iguez}
\address{}
\email{rodmiga@yahoo.com}
\title{Asymptotic distribution of the roots of the Ehrhart polynomial of the cross-polytope}
\begin{abstract}
We use the method of steepest descents to study the root distribution of the Ehrhart polynomial of the $d$-dimensional cross-polytope, namely $\mathcal{L}_{d}$,  as $d\rightarrow \infty$.  We prove that the distribution function of the roots, approximately, as $d$ grows,  by variation of argument of the generating function $\sum_{m\geq 0}\mathcal{L}_{d}(m)t^{m+x-1}=(1+t)^{d}(1-t)^{-d-1}t^{x-1}$, as $t$ varies  appropriately on the segment of the imaginary line contained inside the unit disk.  
\end{abstract}
\maketitle
\section{Introduction}
There has been of recent some interest in the root distribution of Ehrhart polynomials.  In \cite{beck-2005-374}, Beck et al.\ come up with bounds for the coefficients and roots of Ehrhart polynomials.  They found them to be contained inside a disk of radius $d!$ centered at $-1/2$, where $d$ is the degree of the Ehrhart polynomial in question.  This result was subsequently improved upon by Braun in \cite{braun-2008-39} and \cite{braun-2008}, and Bey et al.\ in \cite{bey-2006}, where the radius of the above mentioned disk is reduced to the order of $d^{2}$.  All of these last results depend Stanley's non-negativity theorem, which states that the coefficients of the Ehrhart polynomial are, when the polynomial is expressed in a binomial basis, positive.  In \cite{BuChKuVa00}, Vaaler et al.\ prove that the roots of the Ehrhart polynomial of the cross-polytope all have real part $-1/2$.  In \cite{MR1896405}, Rodr\'iguez-Villegas generalizes the last result by proving that if the generating function $\sum_{m\geq 0}\mathcal{P}(m)t^{m}$, where $\mathcal{P}$ is some polynomial, has all of its roots on the unit circle, then, $\mathcal{P}$ has all of its roots on a vertical line in the complex plane.  These last results depend, not on the positivity of the coefficients, but on functional equations that are satisfied by the above generating function.  Here, we look at the zero distribution of the Erhart polynomial of the $d$-th dimensional cross-polytope, namely $\mathcal{L}_{d}$, as $d$ grows.  The form of the generating function allows us the use of the method of steepest descents to study the asymptotics of $\mathcal{L}_{d}$ on the line $\{ z\in \mathbb{C}:\Re{z}=-1/2\}$.  Since our aim is to obtain an asymptotic formula which approximates $\mathcal{L}_{d}$ uniformly on ever large segments of a vertical line, we will have to deal with eventualities that arise in the use of the method of steepest descent for uniform approximations; we will have to deal with coalescing saddle points and with a saddle point coalescing with a singularity of the integrand.  The terminology will all be explained below.   

The \emph{$d$-dimensional cross-polytpe} is the convex hull in $\mathbb{R}^{d}$ of $\{ \pm e_{j}| j=1,\ldots ,d\}$ where $e_{j}=(\delta_{j1},\ldots ,\delta_{jd} )$ where $\delta_{ji}$ is the kronecker delta.  Let $\mathcal{L}_{d}$ be the function $\mathcal{L}_{d}(m)=\# (\mathbb{Z}^{d}\cap m\mathcal{P}_{d})$ for $m\in \mathbb{Z}^{\geq 0}$.  $\mathcal{L}_{d}$ is actually a polynomial, called the \emph{Ehrhart polynomial of the $d$-dimensional cross-polytope}.  For introductory material on the cross-polytope and Ehrhart polynomials see \cite{Mathias-Beck:2007ek}.  $\mathcal{L}_{d}$ can be expressed as a sum of binomials:
\begin{equation}\label{ehrhart}
\mathcal{L}_{d}(x)=\sum_{j=0}^{d}\binom{d}{j}\binom{d-j+x}{d}
\end{equation}
The sequence $\{ \mathcal{L}_{d}(m)\}_{m\geq 0}$ has generating function:
\begin{equation}
\mathcal{E}_{d}(t)=\sum_{m\geq 0}\mathcal{L}_{d}(m)t^{m}=\frac{(1+t)^{d}}{(1-t)^{d+1}}
\end{equation}
We can also write:
\begin{equation}\label{integral}
\mathcal{L}_{d}(x)=\frac{1}{2\pi i}\int_{\Gamma}\mathcal{E}_{d}(t)t^{-x-1}dt,
\end{equation}
where $\Gamma \subset \mathbb{C}$ is some simple closed curve, oriented counterclockwise with respect to the bounded connected component of its complement.  $\Gamma$ contains 1 in the bounded connected component of its complement, while 0 is in the unbounded connected component.  We can derive the above formula as follows.
\[
\begin{split}
\int_{\Gamma}\mathcal{E}_{d}(t)t^{-x-1}dt&=\int_{\Gamma}\frac{(1+t)^{d}}{(1-t)^{d+1}}t^{-x-1}dt\\
&=\sum_{j=0}^{d}\binom{d}{j}\int_{\Gamma}\frac{t^{j-x-1}}{(1-t)^{d+1}}dt\\
&=\sum_{j=0}^{d}\binom{d}{j}2\pi i\mbox{Res} \left[ \frac{t^{j-x-1}}{(1-t)^{d+1}};1\right]\\
&=2\pi i\sum_{j=0}^{d}\binom{d}{j}(-1)^{d+1}\binom{j-x-1}{d}\\
&=2\pi i\sum_{j=0}^{d}\binom{d}{j}\binom{d-j+x}{d}\\
&=2\pi i\mathcal{L}_{d}(x)
\end{split}
\]
The above integral representation allows us the use of the method of steepest descents to produce an asymptotic formula for $\mathcal{L}_{d}$.  An explanation of the method of steepest descents and our use of it will be given in the sections below.  In these notes we use Landau's asymptotic notation.  If $f$ and $g$ are complex valued functions with domain $D \subset \mathbb{C}$, we write $f(z)=O(g(z))$, or $f(z)\ll g(z)$, as $z\rightarrow a$, if there is some constant $C>0$ such that $|f(z)|\leq C|g(z)|$ for all $z$ sufficiently close to $a$ in $D$ (when $a$ is $\infty$, by all $z$ sufficiently close to $a$ we mean for all sufficiently large $|z|$).  The previous situation can also be described by the symbols $g(z)=\Omega (f(z))$ or $g(z)\gg f(z)$ as $z\rightarrow a$.  Sometimes we will omit ``$z\rightarrow a$'' if we believe that its presence is obvious from the context.  By $f(z)=o(g(z))$ as $z\rightarrow a$ we mean $\frac{f(z)}{g(z)}\rightarrow 0$ as $z\rightarrow a$.   By $f(z)=\omega (g(z))$ as $z\rightarrow a$ we mean $|\frac{f(z)}{g(z)}|\rightarrow +\infty$ as $z\rightarrow a$.  By $f(z)\sim g(z)$ as $z\rightarrow \infty$ we understand $\left| \frac{f(z)}{g(z)}\right| \rightarrow 1$ as $z\rightarrow a$.  We will use $f(z)=\Theta [g(z)]$ to denote that $g(z)\ll f(z)\ll g(z)$ as $z\rightarrow a$.  We will use $\mathbb{D}$ to denote the closed unit disk.  Our asymptotic formula for $\mathcal{L}_{d}$ is given in the following theorem.
\begin{thm} \label{main theorem}
Let $\epsilon \in (0,1/2)$.  Let $S_{d}$ be the sequence of sets $S_{d}=\{ z\in \mathbb{C}:\Re{z}\in [\epsilon ,1-\epsilon ],|z|\leq d-\sqrt[6]{d}\}$.  Given $x\in S_{d}$, define $\alpha =\alpha (x)=-i\left( \frac{d}{\Im{x}}-\sqrt{\frac{d^{2}}{\Im{x}^{2}}-1}\right)$.  Let $C_{m}/m$ be the $m$-th degree coefficient of the Taylor series of $\log \{ \mathcal{E}_{d}(t)(1-t)t^{i\Im{x}}\}$ around $\alpha$.  Then, there is a complex valued function $\mathcal{F}(d,x)$ defined for $x\in S_{d}$, such that, for large enough $d$:
\begin{equation}\label{asymptotic formula}
\mathcal{L}_{d}(-x)=\mathcal{F}(d,x)+(-1)^{d+1}\overline{\mathcal{F}(d,1-\overline{x})},
\end{equation}
and,
\begin{equation}
\mathcal{F}(d,x)=\begin{cases}
\frac{\sin \pi x}{\pi}[(2d+1)^{-x}\Gamma(x)+O[(2d+1)^{-x-1}]&\text{; if $x=O( \sqrt{\log d})$}\\
\frac{3C_{2}}{2C_{3}}\mathcal{E}_{d}(\alpha )\alpha^{x}\mathcal{I}_{d}(x)\{ 1+O(\Im{x}^{-1/28})\}&\text{; if $\sqrt{\log d}\ll |x|\ll d-\sqrt[6]{d}$}
\end{cases}
\end{equation}
Where $D$ is some positive constant, and,
\[
\mathcal{I}_{d}(x)=\int_{-1/3}^{\infty}\exp \left\{ i\frac{9}{4}\frac{C_{2}^{3}}{C_{3}^{2}}F(T)\right\} U'(T)dT
\]
,where $F(T)=2T^{2}\frac{(2T+1)^{2}\sqrt{T+1}}{(3T+1)^{\frac{3}{2}}}$, $U(T)=T( \sqrt{\frac{T+1}{3T+1}}+i)$.
\end{thm}
\begin{proof}
The roles that $F(T)$, $U(T)$ play in our use of steepest descent, and, in our study of the roots of $\mathcal{L}_{d}$ will become clear in the proofs, which will all be given below. 
\end{proof}
The above result might be stated briefly as: ``The function $\mathcal{F}_{d}$ is approximately given by the generating series $\mathcal{E}_{d}(t)t^{x}$ evaluated at the saddle point $\alpha (x)$''.  We will be able to determine the root distribution of $\mathcal{L}_{d}$ on the line $\{ z\in \mathbb{C}: \Re{z}=-1/2\}$ just by measuring the variation of the argument $\mathcal{E}_{d}(t)t^{x}$, for $t$ varying in the intersection of the imaginary axis and $\mathbb{D}$.  Theorem \ref{main theorem} only refers to the asymptotic nature of $\mathcal{L}_{d}$ in a restricted range which, as we will see below, contains all the roots. 

\section{Asymptotic distribution of roots of $\mathcal{L}_{d}$}

\subsection{Roots with imaginary part in $[-\sqrt{\log d},\sqrt{\log d}]$}

We first study the roots of $\mathcal{L}_{d}$ , which have imaginary part in $[-\sqrt{\log d},\sqrt{\log d}]$.  For large enough $d$, and for $x=1/2+i\tau$ with $\tau \in [0,\sqrt{\log{d}}]$, we can write, using (\ref{near zero formula}):
\[
\frac{\pi}{\sin \pi x}\mathcal{L}_{d}(-x)=\begin{cases}
2\Re{(2d+1)^{-x}\Gamma (x)g_{d}(x)} &\text{, if $d$ is even}\\
2i\Im{(2d+1)^{-x}\Gamma (x)g_{d}(x)} &\text{, if $d$ is odd}
\end{cases}
\]
where $g_{d}(x)=1+O\left( \frac{x}{\log{d}}\right)$.  Notice that $g_{d}$ is analytic in $S_{d}=\{ x:\Re{x}\in [\epsilon ,1-\epsilon ], \Im{x}\in [-\sqrt{\log d},0]\}$.  We already know the roots of $\mathcal{L}_{d}$ have real part $-1/2$.  Hence, to find roots of $\mathcal{L}_{d}(-x)$, we solve:
\begin{equation}  \label{root eq}
\text{arg} \left\{ (2d+1)^{-x}\Gamma (x)g_{d}(x)\right\} =\begin{cases}
\frac{\pi}{2}\mod{\pi} &\text{, if $d$ is even}\\
0\mod{\pi} &\text{, if $d$ is odd}
\end{cases}
\end{equation}
Then, using Stirling's formula, as $d\rightarrow \infty$,  
\[
\begin{split}
\text{arg} \left\{ (2d+1)^{-x}\Gamma (x)g_{d}(x)\right\} &=\\
&=\text{arg} \left\{ (2d+1)^{-x}\Gamma (x)\right\} +\text{arg}[g_{d}(x)]\\
&=\tau \log \frac{|x|}{e(2d+1)}+O(1)\\
\end{split}
\]
Here we take the principal branch of the argument function; that for which $\text{arg}(1)=0$.

Now, if $f=u+iv$ is some analytic function with real and imaginary parts $u,v$, we can write the partial derivatives of the component functions $u$ and $v$ in terms of the derivative of $f$.  In fact, if $f$ is a function of the complex variable $x=\sigma +i\tau$, with $\sigma$ and $\tau$ in $\mathbb{R}$, then, using the Cauchy-Riemann equations for example,
\[
\frac{d}{d\tau}v(x)=\Re{f'(x)} 
\]
This means that we can use Cauchy's theorem to bound the partial derivatives of $u$ and $v$ in terms of $f$.  Then, since $g_{d}$ is analytic and different from $0$ inside $S_{d}$, we can use Cauchy's estimate (see \cite{ahlfors1979complex}) to obtain a uniform bound for the derivate of the $g_{d}$ for all $d$.  So, for large enough $d$,
\[
\begin{split}
\frac{d}{d\tau} \text{arg} \left\{ (2d+1)^{-x}\Gamma (x)g_{d}(x)\right\} &= \frac{d}{d\tau} (2d+1)^{-x}\Gamma (x)+ \frac{d}{d\tau}g_{d(x)} \\
&= \frac{d}{d\tau}\left[ \tau \log \frac{|x|}{e(2d+1)}\right] +O(1)\\
&=\log (2d+1) +O(1)<0,
\end{split}
\]
In the estimates above we used that $|x|\leq \log d$.  We also took the argument to be the principal branch with $\text{arg}=0$ on the positive real line.  It follows that for large enough $d$, $\text{arg} \left\{ (2d+1)^{-x}\Gamma (x)g_{d}(x)\right\} $ is a monotone function of $\tau$ for $x\in S$ (this method is used by Lagarias to count the zeros of differenced $L$ functions in \cite{lagarias-2005-120}).  Let $a$ and $b$ be such that $-\log{d}\leq a<b\leq \log{d}$.  Then, by (\ref{root eq}) and the work done above, as $d\rightarrow \infty$, 
\begin{multline} \label{small root counter}
\# \{ x\in S_{d}:\text{$x$ is root of $\mathcal{L}_{d}$ with $\Im{x} \in [a,b]$}\} =\\
=\frac{1}{\pi}\text{arg} \left\{ (2d+1)^{-\frac{1}{2}-ia}\Gamma \left( \frac{1}{2}+ia\right) \right\} -\frac{1}{\pi}\text{arg} \left\{ (2d+1)^{-\frac{1}{2}-ib}\Gamma \left( \frac{1}{2}+ib\right) \right\}  +O_{S}(1)
\end{multline}
Where the error above is actually less than 1.Then, as $d\rightarrow +\infty$:
\[
\# \{ x\in S_{d}:\text{$x$ is root of $\mathcal{L}_{d}$ with $\Im{x} \in [a,b]$}\} \sim \frac{\log d}{\pi}(b-a)\quad ,\mbox{as $d\rightarrow \infty$}
\]

\subsection{Roots with absolute value in $[\sqrt{\log d}, d-\sqrt[6]{d}]$}

We now study the roots of $\mathcal{L}_{d}(-x)$ , which have imaginary part in $[\sqrt{\log d},d-\sqrt[6]{d}]$.  For large enough $d$, and for $x=\frac{1}{2}+i\tau$ with $\tau \in [\sqrt{\log d},d-\sqrt[6]{d}]$, we can write, using (\ref{asymptotic formula}):
\[
\mathcal{L}_{d}(-x)=\begin{cases}
2i\Im{\frac{3C_{2}}{2C_{3}}\mathcal{E}_{d}(\alpha )\alpha^{x}\mathcal{I}_{d}(x)g_{d}(x)}&\text{, if $d$ is even}\\
2\Re{\frac{3C_{2}}{2C_{3}}\mathcal{E}_{d}(\alpha )\alpha^{x}\mathcal{I}_{d}(x)g_{d}(x)} &\text{, if $d$ is odd}
\end{cases}
\]
where $g_{d}(x)=1+O(\tau^{-1/28})$.  As in the previous section $g_{d}$ is analytic in $S_{d}=\{ x:\Re{x}\in [\epsilon ,1-\epsilon ], \Im{x}\in [-d+\sqrt[6]{d},0]\}$.  Also, as in the previous section, to find roots of $\mathcal{L}_{d}(-x)$, we solve:
\begin{equation}  \label{root equation ii}
\text{arg} \left\{ \frac{3C_{2}}{2C_{3}}\mathcal{E}_{d}(\alpha )\alpha^{x}\mathcal{I}_{d}(x)g_{d}(x)\right\} =\begin{cases}
\frac{\pi}{2}\mod{\pi} &\text{, if $d$ is even}\\
0\mod{\pi} &\text{, if $d$ is odd}
\end{cases}
\end{equation}
Now, $U'(T)$ is in the first quadrant.  In fact, looking at (\ref{derivative U}), it is not hard to see that $\Im{U(T)}=i$, while $\Re{U(T)}\geq 1/\sqrt{3}$.  Then (see (\ref{km}) for the definition of $K_{m}$), 
\begin{equation} \label{I argument bound}
\text{arg}[\mathcal{I}_{d}(x)]= \arctan \frac{\int_{-1/3}^{\infty}\exp \left\{ -\frac{9}{4}\tau \frac{K_{2}^{3}}{K_{3}^{2}}F(T)\right\} dT}{\int_{-1/3}^{\infty}\exp \left\{ -\frac{9}{4}\tau \frac{K_{2}^{3}}{K_{3}^{2}}F(T)\right\} \Re{U'(T)}dT}\in \left( 0, \frac{\pi}{3}\right)
\end{equation}
Then, as $d\rightarrow \infty$,  
\[
\begin{split}
&\text{arg} \left\{ \frac{3C_{2}}{2C_{3}}\mathcal{E}_{d}(\alpha )\alpha^{x}\mathcal{I}_{d}(x)g_{d}(x)\right\} =\\
&=\text{arg}\frac{3C_{2}}{2C_{3}}+(2d+1)\text{ arg}(1+\alpha )+\text{arg}(\alpha^{x})+\text{arg} [\mathcal{I}_{d}(x)]+\text{arg}[g_{d}(x)]\\
&=-(2d+1)\arctan (|\alpha |)+\tau\log |\alpha | +O(1)
\end{split}
\]
Use Cauchy's estimate to obtain, for large enough $d$,
\[
\begin{split}
\frac{d}{d\tau} \text{arg} \left\{ \frac{3C_{2}}{2C_{3}}\mathcal{E}_{d}(\alpha )\alpha^{x}\mathcal{I}_{d}(x)g_{d}(x)\right\} &=\left\{ -(2d+1)\frac{1}{1+|\alpha |^{2}}-\frac{\tau}{|\alpha |}\right\} \frac{d}{d\tau}|\alpha |+O(1)\\
&=\left\{ -(2d+1)\frac{1}{1+|\alpha |^{2}}-\frac{\tau}{|\alpha |}\right\} \frac{d}{\tau \sqrt{d^{2}-\tau^{2}}}|\alpha |+O(1)<0,
\end{split}
\]
\begin{figure}[htp]
\centering
\includegraphics[scale=.5,bb=50 80 800 550]{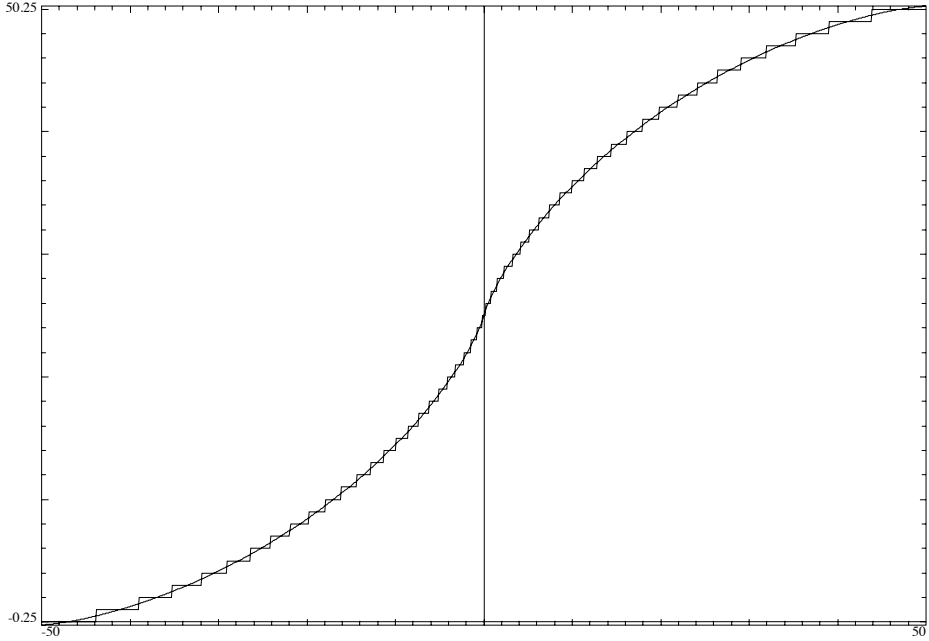}
\caption{Smooth line gives our approximation of the root distribution of $\mathcal{L}_{50}$.  The step function gives the cumulative histogram of the roots of $\mathcal{L}_{50}$.}
\end{figure}
It follows that for large enough $d$, $\text{arg} \left\{ \frac{3C_{2}}{2C_{3}}\mathcal{E}_{d}(\alpha )\alpha^{x}\mathcal{I}_{d}(x)g_{d}(x)\right\} $ is a monotone function of $\tau$.  Now, notice that when $|x|=O(\sqrt{\log d})$, then, $\text{arg} \left\{ \frac{3C_{2}}{2C_{3}}\mathcal{E}_{d}(\alpha )\alpha^{x}\mathcal{I}_{d}(x)\right\} -\text{arg} \left\{ (2d+1)^{-x}\Gamma (x)\right\}=O(1)$.  Then, if $-d+\sqrt[6]{d}<a<b<d-\sqrt[6]{d}$, as $d\rightarrow +\infty$:
\[
\# \{ x\in S_{d}:\text{$x$ is root of $\mathcal{L}_{d}$ with $\Im{x} \in [a,b]$}\} \sim \frac{1}{\pi}\{ \text{arg}[\mathcal{F}(d,1/2-ib)]-\text{arg}[\mathcal{F}(d,1/2-ia)]\},
\]
as $d\rightarrow \infty$.

\begin{figure}[htp]
\centering
\includegraphics[width=\textwidth,,height=5cm,bb=50 80 800 600]{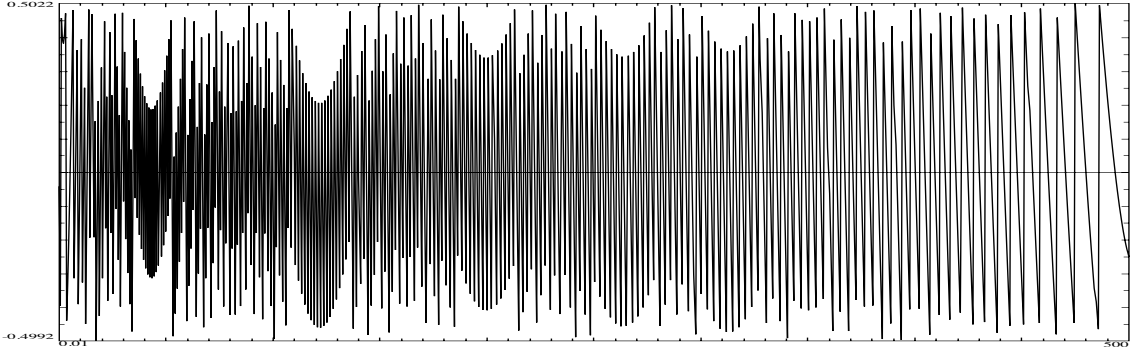}
\caption{Error from the counting function of the roots of $\mathcal{L}_{500}$ on $[0,500]$; i.e.\ $\# \{ x\in S_{500}:\text{$x$ is root of $\mathcal{L}_{500}$ with $\Im{x} \in [0,\tau ]$}\} -\frac{1}{\pi}\{ \text{arg}[\mathcal{F}(500,1/2-i\tau )]-\text{arg}[\mathcal{F}(500,1/2-i0)]\}$ for $\tau \in [0,500].$
}
\end{figure}

\subsection{The root of largest absolute value.}

Again,
\[
\begin{split}
&\text{arg} \left\{ \frac{3C_{2}}{2C_{3}}\mathcal{E}_{d}(\alpha )\alpha^{x}\mathcal{I}_{d}(x)g_{d}(x)\right\} =\\
&=\text{arg}\frac{3C_{2}}{2C_{3}}+(2d+1)\text{arg}(1+\alpha )+\text{arg}(\alpha^{x})+\text{arg} [\mathcal{I}_{d}(x)]+\text{arg}[g_{d}(x)]\\
&=-\frac{\pi}{2}-(2d+1)\arctan (|\alpha |)+\tau\log |\alpha |-\frac{\pi}{4}+\text{arg}[\mathcal{I}_{d}(x)]+O(\tau^{-1/28})
\end{split}
\]
Assume $d$ is odd.  The $d$ even case is similar.  Equation (\ref{root equation ii}) allows us to estimate all the roots of $\mathcal{L}_{d}$ that are of absolute value $O(d-\sqrt[6]{d})$, the range of validity of our approximation.  Assuming that the largest root of $\mathcal{L}_{d}(-x)$ is $x=1/2+i\tau=O(d-\sqrt[6]{d})$, we could obtain an estimate by solving $\text{arg} \left\{ \frac{3C_{2}}{2C_{3}}\mathcal{E}_{d}(\alpha )\alpha^{x}\mathcal{I}_{d}(x)g_{d}(x)\right\} =-[(d-1)/2]\pi$, or, 
\begin{equation} \label{largest root equation}
-(2d+1)\arctan (|\alpha |)+\tau\log |\alpha |=-\frac{2d-5}{4}\pi-\text{arg}[\mathcal{I}_{d}(x)]+O(\tau^{-1/28})
\end{equation}
In fact, $\mathcal{L}_{d}$ is a polynomial with real coefficients and $d$ distinct roots.  This last being a consequence of the monotony of $\text{arg} \left\{ \frac{3C_{2}}{2C_{3}}\mathcal{E}_{d}(\alpha )\alpha^{x}\mathcal{I}_{d}(x)g_{d}(x)\right\}$.  So if we manage to solve (\ref{largest root equation}) in the desired range, we will have found the largest root.  By (\ref{alpha}), if $\tau =o(d)$ as $d\rightarrow \infty$, then, the RHS of the above equation is $o(d)$.  So it must be that $\tau =\Omega (d)$.  Assume, for the moment, that $\tau =d-f(d)$, where $f(d)=\Omega (\sqrt[6]{d})$ and positive.  Let $y=\sqrt{f(d)/d}$.  Then, using (\ref{alpha}), we may write (\ref{largest root equation}) in terms of $y$:
\[
\frac{1}{\sqrt{2}}y+\frac{16d+1}{\sqrt{288}}y^{3}+O(dy^{5})=\frac{3}{2}\pi -\text{arg}[\mathcal{I}_{d}(x)]+O(d^{-1/28})
\]
By (\ref{I argument bound}), $\frac{3}{2}\pi -\text{arg}[\mathcal{I}_{d}(x)]$ is positive and bounded away from zero.  It then follows that:
\[
\frac{4d}{3\sqrt{2}}y^{3}=\frac{3}{2}\pi -\text{arg}[\mathcal{I}_{d}(x)]+O(d^{-1/28})
\]
,or,
\[
f(d)=\sqrt[3]{d}\left( \frac{9}{8}\right)^{1/3}\left\{ \frac{3}{2}\pi -\text{arg}[\mathcal{I}_{d}(x)]+O(d^{-1/28})\right\}^{2/3}
\]
In any case we have shown that the largest root has imaginary part $\tau$, such that $d-\tau =\Theta (\sqrt[3]{d})$. 
\begin{table}[ht]
\caption{$\tau$ is the imaginary part of the largest root of $\mathcal{L}_{d}$} 
\centering  
\begin{tabular}{c c c } 
\hline\hline                        
d & $\tau$ & $\frac{d-\tau}{\sqrt[3]{d}}$ \\ [0.5ex] 
\hline                  
100 & 91.9987057014 & 1.7238266002\\ 
200 & 189.7372321215 & 1.7549086218\\ 
300 & 288.1562327578 & 1.7692238245\\ 
400 & 386.8992027271 & 1.7780517454\\ 
500 & 485.8385218444 & 1.7842344425\\ 
600 & 584.9118679958 & 1.7888958567\\ 
700 & 684.0835726177 & 1.7925842603\\ 
800 & 783.3310874715 & 1.7956041698\\ 
900 & 882.6391445854 & 1.7981404758\\ 
1000 & 981.9968699646 & 1.8003130035\\ [1ex]      
\hline 
\end{tabular}
\label{table:nonlin} 
\end{table}

\section{Proof of the asymptotic formula.}

Let $\{ x_{d}\}_{d\geq 0}$ be a sequence of complex number such that $x_{d}=\sigma_d +i\tau_{d}$, where $\sigma_d, \tau_{d} \in \mathbb{R}$.  Let $\epsilon \in (0,1)$.  Assume $\sigma_{d}\in [\epsilon ,1-\epsilon ]$ and $|\tau_{d} |\in [0,d-\sqrt[6]{d}]$.  In this section we use the method of steepest descents to study the asymptotic value of:
\begin{equation}
\int_{\Gamma}\mathcal{E}_{d}(t)t^{-x_{d}-1}dt
\end{equation}
as $d\rightarrow +\infty$, where $\Gamma$ is as in (\ref{integral}).  The basic result needed for the implementation of the method is Watson's lemma, which we subsequently introduce.  Most of the method of steepest descent consists in modifying integration contours in a way that allows for the use of the lemma.  From now on we drop the subindex in most of the quantities that vary with $d$.  We hope this last will make our notation clearer.

\section{Watson's lemma} \label{Asymptotic lemma}

A basic result needed for the application of the method of steepest descents is Watson's Lemma (see \cite{Olver74}), which deals with integrals of the form:
\begin{equation} \label{basic laplace}
\int_{0}^{a}e^{-X\phi (t)}q(t)dt
\end{equation}
as $X\rightarrow +\infty$, where $\phi(t)$ is some complex valued function of the form $t+O(t^{2})$ as $t\rightarrow 0^{+}$ that has an increasing real part on $[0,a]$, and, $q(t)\sim t^{\alpha-1}$ as $t\rightarrow 0^{+}$.  A typical version is: 
\begin{thm}
Let $q(t)$ be a function of a positive real variable $t$, such that
\[
q(t)\sim \sum_{s\geq 0}a_{s}t^{(s+\lambda-\mu)/\mu}\quad (t\rightarrow 0)
\]
where $\lambda$ and $\mu$ are positive constants. 
\begin{equation} \label{watson}
\int_{0}^{\infty}e^{-xt}q(t)dt\sim \sum_{s\geq 0}\Gamma \left( \frac{s+\lambda}{\mu}\right) \frac{a_{s}}{x^{(s+\lambda)/\mu}}\quad (x\rightarrow \infty)
\end{equation}
provided the integral converges throughout its range for all sufficiently large $x$.
\end{thm}
The proof of Watson's lemma is a consequence of the fact that the larger part of the area under the graph of $e^{-xt}$, lies over ever smaller intervals of the form $[0,\epsilon)$ as $x$ grows.  This implies that the value of the integral in (\ref{watson}) depends mostly on the asymptotic expansion of $q(t)$ when $t$ is close to 0, the main term in the formula coming from:
\[
\int_{0}^{\epsilon}e^{-xt}\sum_{s=0}^{M}a_{s}t^{(s+\lambda)/\mu}dt\sim \int_{0}^{\infty}e^{-xt}\sum_{s=0}^{M}a_{s}t^{(s+\lambda)/\mu}dt=\sum_{s=0}^{M}\Gamma \left( \frac{s+\lambda}{\mu}\right) \frac{a_{s}}{x^{(s+\lambda)/\mu}}
\]

In our results the functions $\phi(t)$ and $q(t)$ in (\ref{basic laplace}) will vary with $x$.  This will require that we take care about how things change as $x$ grows.  The following result can be found in any book on asymptotics, and, is basically a reformulation of Watson's lemma in a form which is convenient for subsequent proofs.  
\begin{lem}\label{low-lying lemma}
Let $\{ B_{n}\}_{n\geq 0}$ be a sequence of positive numbers such that $B_{n}\rightarrow +\infty$.  Let $\{ \phi_{n}(t)\}_{n\geq 0}$ be a sequence of functions, analytic on the unit disk, such that:
\begin{enumerate}
\item Each $\phi_{n}(t)$ is increasing on $[0,1)$.
\item $\phi_{n}(t)=t+O(t^{2})$ as $t\rightarrow 0^{+}$ and uniformly for all $n$.
\end{enumerate}
Let $\epsilon >0$.  Then, uniformly for all $x$ with $\Re{x}\geq \epsilon$:
\begin{equation} \label{main asymptotic}
\int_{0}^{1}e^{-B_{n}\phi_{n}(t)}t^{x-1}dt=B_{n}^{-x}\Gamma (x)+O( B_{n}^{-\Re x-1})
\end{equation}
as $n\rightarrow \infty$.
\end{lem}
\begin{proof}
Since $\phi_{n}(t)\sim t$ uniformly for all $n$ as $t\rightarrow 0$, $\phi_{n}(\alpha )>\alpha /2$ for all small enough $\alpha \in (0,1)$ and for all $n$.  Then, since all $\phi_{n}$ are increasing,
\begin{multline} \label{easy bound}
I_{0}=\int_{\alpha}^{1}e^{-B_{n}\phi_{n}(t)}t^{x-1}dt\ll  \\
\ll e^{-B_{n}\phi_{n}(\alpha )}\int_{\alpha}^{1}t^{\Re x-1}dt\\
\ll O_{\Re x}\{ e^{-\frac{B_{n}\alpha}{2}}\},
\end{multline}
as $n\rightarrow +\infty$.
\begin{equation}
\begin{split}
&\int_{0}^{\alpha}e^{-B_{n}\phi_{n}(t)}t^{x-1}dt=\\
&=\int_{0}^{\alpha}e^{-B_{n}t}e^{B_{n}[t-\phi_{n}(t)]}t^{x-1}dt\\ 
&=\int_{0}^{\infty}e^{-B_{n}t}t^{x-1}dt-\int_{\alpha}^{\infty}e^{-B_{n}t}t^{x-1}dt+\int_{0}^{\alpha}e^{-B_{n}t}\{e^{B_{n}[t-\phi_{n}(t)]}-1\}t^{x-1}dt\\
&=B_{n}^{-x}\Gamma(x)-I_{1}+I_{2}
\end{split}
\end{equation}
Now,
\begin{equation} \label{E1}
|I_{1}|=\left| \int_{\alpha}^{\infty}e^{-B_{n}t}t^{x-1}dt\right| \leq e^{-(B_{n}-1)\alpha}\int_{\alpha}^{\infty}e^{-t}t^{\Re x-1}dt\ll _{\Re x}e^{-B_{n}\alpha} \\,
\end{equation}
as $n\rightarrow +\infty$.  

By item (2) in the statement of the lemma, $f(t)=e^{B_{n}O(t^{2})}$ for $t=o(B_{n}^{-\frac{1}{2}})$ as $n\rightarrow \infty$, and uniformly for all $n$.  Then, there is $C>0$ such that $f(t)-1\ll e^{CB_{n}t^{2}}-1=O(B_{n}t^{2})$ for $t$ and $n$ as before.  Taking $\alpha =o(B_{n}^{-\frac{1}{2}})$ as $n\rightarrow \infty$,  
\begin{equation}
\begin{split} \label{E2}
I_{2}&=\int_{0}^{\alpha}e^{-B_{n}t}\{e^{B_{n}[t-\phi_{n}(t)]}-1\}t^{x-1}dt\\
&\ll \int_{0}^{\alpha}e^{-B_{n}t}B_{n}t^{\Re x+1}dt\\
&\ll B_{n}^{-\Re x-1}\Gamma (\Re x)  
\end{split}
\end{equation}
as $n\rightarrow \infty$.

Finally, take $\alpha =(\Re x+1)\frac{\log{B_{n}}}{B_{n}}$ in all the above estimates (notice that $\alpha$ is still smaller than $\frac{1}{\sqrt{B_{n}}}$ for large enough $n$).  We get that, as $n\rightarrow \infty$,
\begin{equation}
I_{0},I_{1},I_{2}\ll B_{n}^{-\Re x-1}
\end{equation}
Putting all the estimates together:
\begin{equation}
\int_{0}^{1}e^{-B_{n}\phi_{n}(t)}t^{x-1}dt=B_{n}^{-x}\Gamma(x)+O_{\Re x}\{ B_{n}^{-\Re x-1}\},
\end{equation}
as $n\rightarrow \infty$.
\end{proof}

\subsection{The case $|\tau |\in [0,\log {d}]$}
We begin by restricting $\tau$ to interval $[0,\log {d}]$.  To obtain all the subsequent results in the $\tau <0$ case, we just have to notice that $\mathcal{L}_{d}(\overline{x})=\overline{\mathcal{L}_{d}(x)}$.  A direct application of Lemma \ref{low-lying lemma} is enough to deal with the present case.  We will use a Mellin transform representation: 
\begin{prop} 
If $\Re{x}<0$, then,
\[ \label{mellin formula}
\mathcal{L}_{d}(x)=-\frac{\sin \pi x}{\pi}\int_{0}^{\infty}\mathcal{E}_{d}(-t)t^{-x-1}dt
\]
\end{prop}
\begin{proof}
Let $B(x,y)=\frac{\Gamma (x)\Gamma (y)}{\Gamma (x+y)}$ be the Beta function.  Let $(a)_{j}=\frac{\Gamma (a+j)}{\Gamma (a)}$ be the Pochhammer symbol.  Then, by well known formulas,
\[
\begin{split}
\int_{0}^{\infty}\mathcal{E}_{d}(-t)t^{-x-1}dt&=\int_{0}^{\infty}\frac{(1-t)^{d}}{(1+t)^{d+1}}t^{-x-1}dt\\
&=\sum_{j=0}^{d}\binom{d}{j}(-1)^{j}\int_{0}^{\infty}\frac{t^{j-x-1}}{(1+t)^{d+1}}dt\\
&=\sum_{j=0}^{d}\binom{d}{j}(-1)^{j}B(j-x,d+1-x)\\
&=\sum_{j=0}^{d}\binom{d}{j}(-1)^{j}\frac{\Gamma (j-x)\Gamma (d+1-j+x)}{\Gamma (d+1)}\\
&=\Gamma (-x)\Gamma(1+x)\sum_{j=0}^{d}\binom{d}{j}(-1)^{j}\frac{(-x)_{j}(1+x)_{d-j}}{d!}\\
&=-\frac{\pi}{\sin \pi x}\sum_{j=0}^{d}\binom{d}{j}\binom{d+x-j}{d}\\
&=-\frac{\pi}{\sin \pi x}\mathcal{L}_{d}(x)
\end{split}
\]
\end{proof}
Using the functional equation:
\begin{equation} \label{functional}
\mathcal{E}_{d}(t^{-1})=(-1)^{d+1}t\mathcal{E}_{d}(t)
\end{equation}
We write:
\[
\begin{split}
\frac{\pi}{\sin \pi x}\mathcal{L}_{d}(-x)&=\int_{0}^{\infty}\mathcal{E}_{d}(-t)t^{x-1}dt\\
&=\int_{0}^{1}\mathcal{E}_{d}(-t)t^{x-1}dt+\int_{1}^{\infty}\mathcal{E}_{d}(-t)t^{x-1}dt\\
&=\int_{0}^{1}\mathcal{E}_{d}(-t)t^{x-1}dt+(-1)^{d+1}\int_{0}^{1}\mathcal{E}_{d}(-t)t^{(1-x)-1}dt
\end{split}
\]
About $\mathcal{E}_{d}(-t)$ we can say the following:
\begin{enumerate}
\item $\mathcal{E}_{d}(-t)$ is decreasing on $[0,1]$ for all $d$.
\item $\mathcal{E}_{d}(-t)=\exp \{ d\log (1-t)-(d+1)\log (1+t)\}=\exp \{ (2d+1)[t+O(t^{2})]\}$ as $t\rightarrow 0+$, uniformly for all $d$.
\end{enumerate}
Then, it is not hard to see that the hypotheses of theorem \ref{low-lying lemma} are satisfied, and, for all $x\in S_{d}$, and uniformly for all large enough $d$,
\begin{multline} \label{near zero formula}
\frac{\pi}{\sin \pi x}\mathcal{L}_{d}(-x)=\\
=(2d+1)^{-x}\Gamma (x)+O[(2d+1)^{-x-1}] +(-1)^{d+1}(2d+1)^{x-1}\Gamma (1-x)+O[(2d+1)^{x-2}]
\end{multline}

\subsection{The case $|\tau |\rightarrow +\infty$, $|\tau |\leq d-\sqrt[6]{d}$}
Next, we will take care of the case $|\tau |\in [\log{d}, d-\sqrt[6]{d}]$.  The method of steepest descents (two references are \cite{Olver74} and \cite{Erdelyi56}) deals with integrals of the form:
\begin{equation} \label{basic steepest}
I(X)=\int_{\Gamma}e^{X\phi (t)}q(t)dt,
\end{equation}
where $\Gamma \subset \mathbb{C}$, $X$ is some complex parameter, $\phi (t)$ and $q(t)$ are fuctions analytic in some domain containing $\Gamma$.  We are interested in what happens to the value of the integral in (\ref{basic steepest}) as $|X|\rightarrow +\infty$.  The idea of the method is to deform $\Gamma$ in such a way that $I(X)$ that is amenable to the application of Watson's lemma.  Roughly stated, the method for achieving this last goes as follows:
\begin{enumerate}
\item Find the saddle points of the function $|\exp [X\phi (t)]|$, that is, where $\phi'(t)=0$.  Assume that $a$ is one of these saddle points, and, the Taylor series of $\phi (t)$ at $a$ is $\phi(a)+\sum_{m\geq M}C_{m}(t-a)^{m}$, with $M\geq 2$ (remember that $\phi'(a)=0$) and $C_{M}\neq 0$.  We know that there are $M$ curves $\Gamma_{1},\ldots ,\Gamma_{M}$ which pass through $a$ which satisfy $\Im{X\phi(t)}=\Im{X\phi (a)}$.  These last are the curves along which the modulus of $e^{X\phi (t)}$ varies the fastest. 
\item Deform $\Gamma$ into a new curve $\Gamma'$, which passes through the saddle points of $|\exp [X\phi (t)]|$.  $\Gamma'$ must also adhere as closely as possible to the ``paths of steepest descent''.  These last are the $\Gamma_{k}$ above for which $\Re{XC_{M}}<0$.
\item We must check that the size of the integral on sections of $\Gamma'$ which are far from the saddle points become small as $|X|$ grows. 
\end{enumerate}
Given an integral of the form (\ref{basic steepest}), it is not at all obvious that the method above can be applied.  There have been some efforts to make the use of the method easier and more automatic.  For a useful discussion on this last topic see \cite{Lopez2009347}.  In this last the authors advice that instead of deforming the curve $\Gamma$ in (\ref{basic steepest}) onto the paths of steepest descent, we should deform $\Gamma$ so that it lies on a straight line tangent to the path of steepest descent, that is, integrate over the paths of steepest descent of $|\exp [XC_{M}(t-a)^{M}]|$, at least when inside the disks of convergence of the Taylor series of $\phi$ around the saddle point.  Outside of these convergence disks we take paths that minimize the absolute value of the integral.  In our case this method would work fine for most values that $\tau <d$ could take.  Problems do arise when $\tau$ is close to $d$.  More specifically, when $d-|\tau |=\Theta (\sqrt[3]{d})$.  This last is the ``coalescing saddle point case'' studied by Chester, Friedman and Ursell in \cite{CambridgeJournals:2048816}.   In this last case, it is easier to apply slightly modified version of the method outlined by Pagola et al. in \cite{Lopez2009347}.  That is, in the vicinity a saddle point $a$, we integrate instead over the path of steepest descent of $\exp \{ X[\sum_{m=M}^{M'}C_{m}(t-a)^{m}]\}$, for some $M'>M$.

We remind the reader that $\tau =\tau_{d}=\Im{-x_{d}}$.  Unless otherwise stated, we assume $\tau >0$.  Let $\mathcal{Q}_{d}(t)=\log \{ (\frac{1+t}{1-t})^{d}t^{i\tau}\}$.  We look for the zeros of $\mathcal{Q}'_{d}$, which are just the zeros of the polynomial $2dt+i\tau (1-t^{2})$.  These last are:
\begin{equation} \label{alpha}
\alpha_{\pm}=-i\left( \frac{d}{\tau}\pm \sqrt{\frac{d^{2}}{\tau^{2}}-1}\right)
\end{equation}
Three things to notice are that
\begin{enumerate}
\item $\alpha_{\pm}$ are pure imaginary.
\item $\alpha_{+}\alpha_{-}=-1$.
\item $i\alpha_{-}\in (0,d)$ and $i\alpha_{-}\ll \frac{\tau}{d}$.
\end{enumerate}
Let $\alpha=\alpha_{-}$.  The Taylor series of $\mathcal{Q}_{d}$ around $\alpha$ is then $\mathcal{Q}_{d}(\alpha )+\sum_{m\geq 2}\frac{C_{m}}{m}(t-\alpha )^{m}$, where:
\[
C_{m}=\frac{(-1)^{m+1}d}{(1+\alpha )^{m}}+\frac{d}{(1-\alpha )^{m}}+\frac{(-1)^{m+1}i\tau_{d}}{\alpha^{m}}
\]
We can estimate the size of the $C_{m}$ as follows.  Remember that $\alpha$ satisfies $2dt+i\tau (1-t^{2})=0$.
\begin{equation} \label{cm bound}
\begin{split}
C_{m}&=d\frac{(-1)^{m+1}(1-\alpha )^{m}+(1+\alpha )^{m}}{(1-\alpha^{2} )^{m}}+\frac{(-1)^{m+1}i\tau}{\alpha^{m}}\\
&=d\frac{(-1)^{m+1}(1-\alpha )^{m}+(1+\alpha )^{m}}{\left( -\frac{2d}{i\tau}\alpha \right)^{m}}+\frac{(-1)^{m+1}i\tau}{\alpha^{m}}\\
&=-\frac{i\tau}{(-\alpha )^{m}}\left\{ 1-\left( \frac{i\tau}{d}\right)^{m-1}\left[ (-1)^{m+1}\left( \frac{1-\alpha}{2}\right)^{m}+\left( \frac{1+\alpha}{2}\right)^{m}\right] \right\}\\
&\ll \frac{\tau}{\alpha^{m}}
\end{split}
\end{equation}
Notice that the bound is uniform in all the $\alpha$ and $d$, since $|\alpha |<1$ and $|\tau |<d$.   Using the above formula, and the fact that $\alpha $ is pure imaginary, we get:
\begin{equation}\label{km}
C_{m}=-(-i)^{m-1}\frac{\tau}{|\alpha|^{m}}K_{m},
\end{equation}
where $K_{m}$ is inside the interval $(0,2)$.  In particular,
\begin{equation} \label{k2}
K_{2}=\sqrt{1-\frac{\tau^{2}}{d^{2}}}
\end{equation}
Now, for $m\geq 3$,
\begin{equation}\label{c3}
\begin{split}
K_{m}&=1-\left(\frac{i\tau}{d}\right)^{m-1}\left[ (-1)^{m+1}\left( \frac{1-\alpha}{2}\right)^{m}+\left( \frac{1+\alpha}{2}\right)^{m}\right] \\
&=1-\left(\frac{i\tau}{d}\right)^{m-1}\begin{cases}
2\Re{\left( \frac{1+\alpha}{2}\right)^{m}};&\text{$m$ odd}\\
2i\Im{\left( \frac{1+\alpha}{2}\right)^{m}};&\text{$m$ even}
\end{cases}\\
&\geq 1-\left(\frac{\tau}{d}\right)^{m-1}2^{1-\frac{m}{2}}\\
&\geq 1-\frac{1}{\sqrt{2}}
\end{split}
\end{equation}
Here we used that $0<\tau <d$ and $0<i\alpha <1$.  Therefore, it is not difficult to see that $C_{3}\gg \frac{\tau}{|\alpha |^{3}}$, uniformly for all $\tau$.  At the same time, $C_{2}=o(\frac{\tau}{|\alpha |^{2}})$ whenever $\tau \sim d$ as $d\rightarrow \infty$ by (\ref{k2}).  These last two facts determine our choice of integration contours.  

Since $\mathcal{E}_{d}$ satisfies the functional equation (\ref{functional}), and is real on the real line, we can restrict ourselves to working inside $\mathbb{D}$.  This last will become clearer as we go through the calculations.  We first define the restriction of our integration path to $\mathbb{D}$:
\[
\Gamma_{d}\cap \mathbb{D}=\Delta_{d}
\]
\begin{figure}[htp]
\centering
\includegraphics[scale=.5,bb=50 10 300 300]{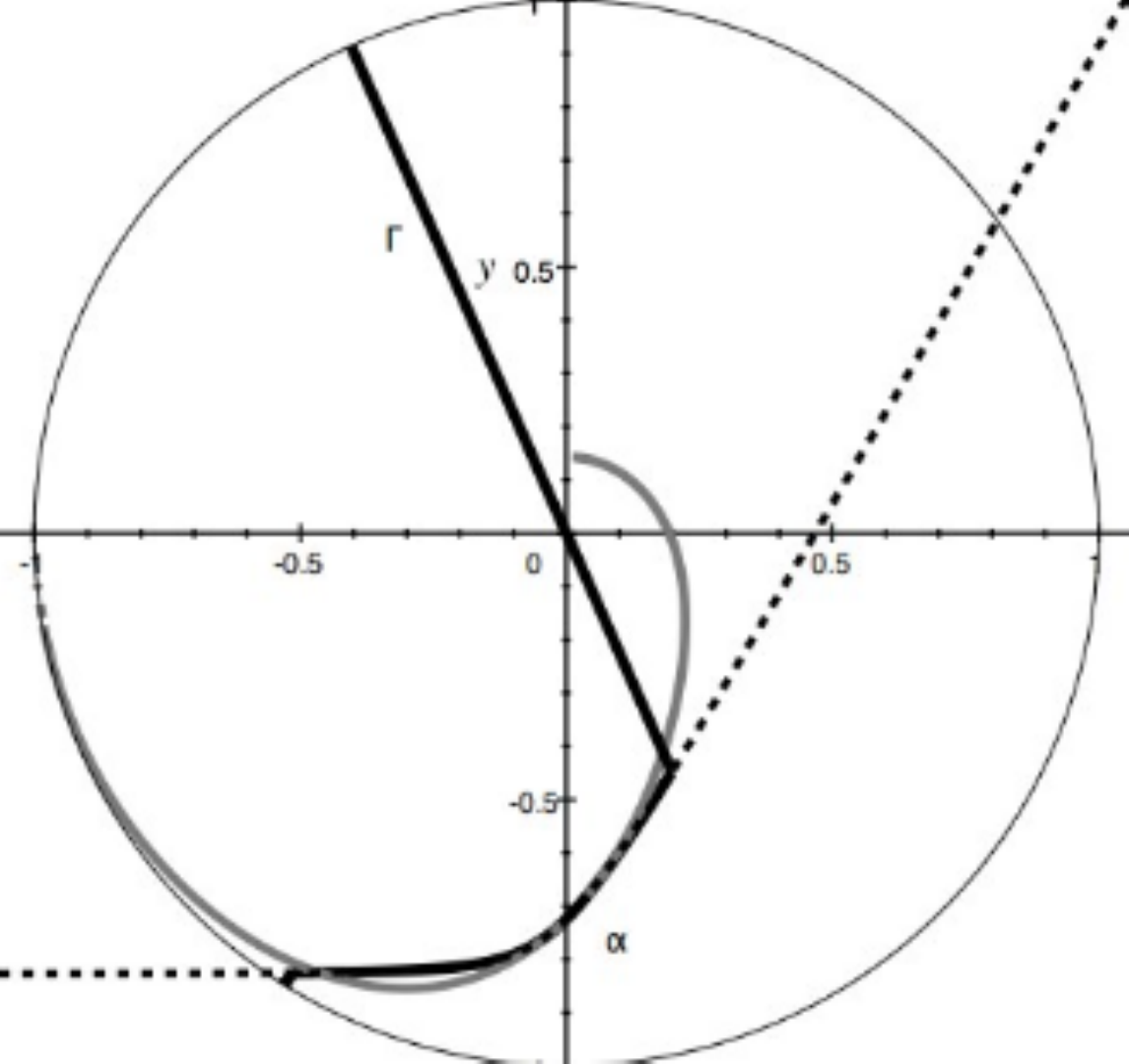}
\caption{Integration path $\Gamma_{d}$ restricted to $\mathbb{D}$ in solid black.  The actual steepest descent path is in gray.  The cubic through the saddle point $\alpha$ and tangent to the steepest descent path in dashes.}
\end{figure}
The reason for the subscript $d$ is that the path will actually vary with $d$.  First, we define the section of path which passes through the saddle point $\alpha$ and is tangent to path of steepest descent that goes through $\alpha$.  We choose the path of integration in such a way that $\Im{\frac{C_{2}}{2}(t-\alpha )^{2}+\frac{C_{3}}{3}(t-\alpha )^{3}}=0$ and $\Re{\frac{C_{2}}{2}(t-\alpha )^{2}+\frac{C_{3}}{3}(t-\alpha )^{3}}<0$.  Now, from (\ref{k2}) and (\ref{km}) we have:
\[
\frac{C_{2}}{C_{3}}=i|\alpha |\frac{K_{2}}{K_{3}}
\]
We finally define:
\[
\Delta_{\alpha}(T)=\alpha +i\frac{3C_{2}}{2C_{3}}T\left( \sqrt{\frac{T+1}{3T+1}}+i\right) =\alpha +\frac{3}{2}|\alpha |\frac{K_{2}}{K_{3}}T\left( \sqrt{\frac{T+1}{3T+1}}+i\right) ,
\]
We will determine the domain of $\Delta_{\alpha}$ below.  This last will be a subset of $(-\frac{1}{3},\infty )$, below.  We thus have:
\[
\begin{split}
&\frac{C_{2}}{2}(\Delta_{\alpha}(T)-\alpha )^{2}+\frac{C_{3}}{3}(\Delta_{\alpha}(T)-\alpha )^{3}=\\
&=\frac{C_{2}}{2}\left( \frac{3}{2}|\alpha |\frac{K_{2}}{K_{3}}\right)^{2}T^{2}\left( \sqrt{\frac{T+1}{3T+1}}+i\right)^{2}+\frac{C_{3}}{3}\left( \frac{3}{2}|\alpha |\frac{K_{2}}{K_{3}} \right)^{3}T^{3}\left( \sqrt{\frac{T+1}{3T+1}}+i\right)^{3}\\
&=\frac{9}{8}\tau \frac{K_{2}^{3}}{K_{3}^{2}}T^{2}\left[ i\left( \sqrt{\frac{T+1}{3T+1}}+i\right)^{2}+T\left( \sqrt{\frac{T+1}{3T+1}}+i\right)^{3}\right] \\
&=-\frac{9}{4}\tau \frac{K_{2}^{3}}{K_{3}^{2}}T^{2}\frac{(2T+1)^{2}\sqrt{T+1}}{(3T+1)^{\frac{3}{2}}}
\end{split}
\]
From the above equation define:
\begin{equation}\label{F}
F(T):=2T^{2}\frac{(2T+1)^{2}\sqrt{T+1}}{(3T+1)^{\frac{3}{2}}}
\end{equation}
$F(T)$ is positive on $(-\frac{1}{3},\infty )$, decreasing on $(-\frac{1}{3},0)$ and increasing on $[0,\infty )$.  To see this last just write:
\[
F'(T)=F(T)\frac{F'(T)}{F(T)}=F(T)\left\{ \frac{2}{T}+\frac{4}{2T+1}+\frac{\frac{1}{2}}{T+1}-\frac{\frac{9}{2}}{3T+1}\right\},
\] 
and, use the fact that $2T+1,T+1>3T+1>0$ when $T<0$, while $0<2T+1,T+1\leq 3T+1$ when $T\geq 0$.  

Write $\phi [\Delta_{\alpha}(T)]=-\frac{9}{4}\tau \frac{K_{2}^{3}}{K_{3}^{2}}F(T)-G(T)$.  We want $\Re{\frac{9}{4}\tau \frac{K_{2}^{3}}{K_{3}^{2}}F(T)+G(T)}$ to grow like $F(T)$ in the domain of definition of $\Delta_{\alpha}(T)$.   
Assuming we have chosen the domain of $\Delta_{\alpha}(T)$, we have:
\[
\begin{split}
&\int_{\Delta_{\alpha}}\mathcal{E}_{d}(t)t^{x-1}dt=\\
&=\int_{\Delta_{\alpha}}\exp \{ \mathcal{Q}(t)\}\frac{t^{\sigma -1}}{1-t}dt\\
&=\int_{J}\exp \left\{ \mathcal{Q}[\Delta_{\alpha}(T)]\right\}\frac{[\Delta_{\alpha}(T)]^{\sigma -1}}{1-\Delta_{\alpha}(T)}\Delta_{\alpha}'(T)dt\\
&=\frac{3}{2}|\alpha |\frac{K_{2}}{K_{3}}\mathcal{E}_{d}(\alpha )\alpha^{x-1}\int_{J}\exp \left\{ -\frac{9}{4}\tau \frac{K_{2}^{3}}{K_{3}^{2}}F(T)-G(T)\right\} \frac{\left[ 1+\frac{3}{2}\frac{|\alpha |}{\alpha}\frac{K_{2}}{K_{3}}U(T)\right]^{\sigma -1}}{1-\frac{3}{2}\frac{|\alpha |}{1-\alpha}\frac{K_{2}}{K_{3}}U(T)}U'(T)dT\\
&=\frac{3}{2}|\alpha |\frac{K_{2}}{K_{3}}\mathcal{E}_{d}(\alpha )\alpha^{x-1} \tilde{\mathcal{I}}_{d}(x)
\end{split}
\]
Now, from (\ref{km}),
\[
\begin{split}
G(T)&=\sum_{m\geq 4}\frac{C_{m}}{m}\left[ \Delta_{\alpha}(T)\right]^{m}\\
&=\sum_{m\geq 4}-(-i)^{m-1}\frac{\tau}{m|\alpha|^{m}}K_{m}\left[ \frac{3}{2}|\alpha |\frac{K_{2}}{K_{3}}T\left( \sqrt{\frac{T+1}{3T+1}}+i\right) \right]^{m}\\
&=\sum_{m\geq 4}-(-i)^{m-1}\frac{\tau K_{m}}{m}\left( \frac{3K_{2}}{2K_{3}}\right)^{m}T^{m}\left( \sqrt{\frac{T+1}{3T+1}}+i\right)^{m}
\end{split}
\]
Let $U(T)=T\left( \sqrt{\frac{T+1}{3T+1}}+i\right)$.  Then, $|U(T)|=|T|\sqrt{\frac{2T+1}{3T+1}}$.  Then, assuming $|U(T)|$ is small enough, for example, if $|U(T)|\leq \frac{K_{3}}{2K_{2}}$,
\begin{equation} \label{asymptotic G}
G(T)= -\sum_{m\geq 4}(-i)^{m-1}\frac{\tau K_{m}}{m}\left( \frac{3K_{2}}{2K_{3}}\right)^{m}U(T)^{m}=-i\frac{\tau K_{4}}{4}\left( \frac{3K_{2}}{2K_{3}}\right)^{4}U(T)^{4}\left\{ 1+O[K_{2}U(T)]\right\}
\end{equation}
Then, for small enough $T$, 
\begin{equation} \label{asymptotic i}
\exp \{ -G(T)\} =1+O\{ \tau K_{2}^{4}U(T)^{4}\}
\end{equation}
Using that $i\alpha \in (0,1)$, and $\sigma$ varies in a compact subset of $(0,1)$, it follows that, for small enough $T$, it is not hard to see that:
\begin{equation} \label{asymptotic ii}
\frac{\left[ 1+\frac{3}{2}\frac{|\alpha |}{\alpha}\frac{K_{2}}{K_{3}}U(T)\right]^{\sigma -1}}{1-\frac{3}{2}\frac{|\alpha |}{1-\alpha}\frac{K_{2}}{K_{3}}U(T)}=1+O\{ K_{2}U(T)\}
\end{equation}
It follows that for small enough $T$, we can write $\tilde{\mathcal{I}}_{d}$ as:
\begin{equation} \label{approximate integral i}
\begin{split}
\tilde{\mathcal{I}}_{d}(x)&=\int_{J}\exp \left\{ -\frac{9}{4}\tau \frac{K_{2}^{3}}{K_{3}^{2}}F(T)\right\} U'(T)\left\{ 1+O\left[ \tau K_{2}^{4}|U(T)|^{4}+K_{2}|U(T)|\right] \right\} dT\\
\end{split}
\end{equation}
We choose $J$ so (\ref{asymptotic G}), (\ref{asymptotic i}) and (\ref{asymptotic ii}) are satisfied;  
\begin{equation} \label{J}
J=\left[ -\frac{1}{3}\left( 1-\frac{1}{1+\tau^{-4/7}K_{2}^{-2}}\right) ,\frac{1}{\tau^{2/7}K_{2}}\right]
\end{equation}
Then, for $T\in J$, $|U(T)|=O(\tau^{-2/7}K_{2}^{-1})=o(\tau^{-1/4}K_{2}^{-1})$.  Now we can write (\ref{approximate integral i}) as:
\begin{equation} \label{approximate integral ii} 
\tilde{\mathcal{I}_{d}}(x)=\int_{J}\exp \left\{ -\frac{9}{4}\tau \frac{K_{2}^{3}}{K_{3}^{2}}F(T)\right\} U'(T)dT\left\{ 1+O\left[ \tau^{-1/28}\right] \right\} 
\end{equation}
This last estimate is sufficient for the purpose of finding the zeros of $\mathcal{L}_{d}$.  The estimate we obtain below is cleaner looking, and also suffices for the zero search.  Now, we estimate the tails.  $F(T)$ is increasing on $[0,\infty )$, and, by (\ref{F}), $F(T)\gg T^{3}$ for all $T\geq 0$.  Also,
\begin{equation} \label{derivative U}
U'(T)=i+T\sqrt{\frac{T+1}{3T+1}}\left\{ \frac{1}{T}+\frac{1/2}{T+1}-\frac{3/2}{3T+1}\right\} =O(1),
\end{equation}
as $T\rightarrow \infty$.  We then have:
\begin{multline}
\int_{\tau^{-2/7}K_{2}^{-1}}^{\infty}\exp \left\{ -\frac{9}{4}\tau \frac{K_{2}^{3}}{K_{3}^{2}}F(T)\right\} U'(T)dT\ll \\
\ll\int_{\tau^{-2/7}K_{2}^{-1}}^{\infty}\exp \left\{ -\tau K_{2}^{3}CT^{3}\right\} dT\\
\ll \frac{1}{3}\int_{\tau^{-6/7}K_{2}^{-3}}^{\infty}\exp \left\{ -\tau K_{2}^{3}CT\right\} T^{-2/3}dT\\
\ll \frac{1}{\tau^{3/7} K_{2}}\exp \{ -C\tau^{1/7}\} 
\end{multline} 
as $d\rightarrow \infty$, for some positive constant $C$ that does not depend on $d$.  Since we assumed $0<d-\tau =\Omega (\sqrt[6]{d})$,we get:
\begin{equation} \label{upper tail}
\int_{\tau^{-2/7}K_{2}^{-1}}^{\infty}\exp \left\{ -\frac{9}{4}\tau \frac{K_{2}^{3}}{K_{3}^{2}}F(T)\right\} U'(T)dT\ll \exp \{ -C\tau^{1/7}\}
\end{equation}
From (\ref{derivative U}) and (\ref{F}), we get:
\begin{equation}\label{lower bound F negative T}
U'(T)\ll (3T+1)^{-3/2}\ll F(T),
\end{equation}
for all $T\in (-\frac{1}{3},0]$.  It follows, then,
\begin{multline} \label{lower tail}
\int_{-1/3}^{-\frac{1}{3}\left( 1-\frac{1}{1+\tau^{-4/7}K_{2}^{-2}}\right) }\exp \left\{ -\frac{9}{4}\tau \frac{K_{2}^{3}}{K_{3}^{2}}F(T)\right\} U'(T)dT\ll \\
\ll\int_{-1/3}^{-\frac{1}{3}\left( 1-\frac{1}{1+\tau^{-4/7}K_{2}^{-2}}\right) }\exp \left\{ -\tau K_{2}^{3}D(3T+1)^{-3/2}\right\} (3T+1)^{-3/2}dT\\
=\frac{9}{2} \int_{[3(1+\tau^{-4/7}K_{2}^{-2})]^{3/2}}^{\infty}\exp \left\{ -\tau K_{2}^{3}DT\right\} T^{-2/3}dT\\
\ll \exp \{ -D\tau^{1/7}\} 
\end{multline} 
as $d\rightarrow \infty$, for some positive constant $D$ that does not depend on $d$.  Putting together (\ref{approximate integral ii}), (\ref{upper tail}), and (\ref{lower tail}), we finally get:
\begin{equation}
\tilde{\mathcal{I}}_{d}(x)=\int_{-1/3}^{\infty}\exp \left\{ -\frac{9}{4}\tau \frac{K_{2}^{3}}{K_{3}^{2}}F(T)\right\} U'(T)dT\{ 1+O[ \tau^{-1/28}]\}, 
\end{equation}
as $d\rightarrow \infty$.
The above equation is the main term of the asymptotic formula we are looking for.  Remember that we started by finding a saddle point of $\phi_{d}$, $\alpha \in\mathbb{D}$.  Then, we integrated $\mathcal{E}_{d}(t)t^{x-1}$ over small curve segment that is tangent to the path of steepest descent through $\alpha$.  This last was $\Delta_{\alpha}(T)=\alpha +\frac{3}{2}|\alpha |\frac{K_{2}}{K_{3}}T\left( \sqrt{\frac{T+1}{3T+1}}+i\right)$.  For our choice of domain $J$ it isn't hard to see that, for $T\in J\cap \mathbb{R}_{>0}$, $\Delta_{\alpha}(T)=\alpha +O(\tau^{-2/7})$ will be contained in the fourth quadrant of the complex plane, while for $T\in J\cap \mathbb{R}_{>0}$, $\Delta_{\alpha}(T)$ will be contained in the third quadrant of the complex plane.  This last will give the main term that comes from the vicinity of the saddle point $\alpha$.  In reality we want to integrate over $-\Delta_{\alpha}$.    So, the contribution to our asymptotic formula coming from the vicinity of $\alpha$ is, as $d\rightarrow \infty$: 
\begin{equation}
\int_{-\Delta_{\alpha}}\mathcal{E}_{d}(t)t^{x_{d}-1}dt=-i\frac{3}{2}\frac{K_{2}}{K_{3}}\mathcal{E}_{d}(\alpha )\alpha^{x}\int_{-1/3}^{\infty}\exp \left\{ -\frac{9}{4}\tau \frac{K_{2}^{3}}{K_{3}^{2}}F(T)\right\} U'(T)dT\left\{ 1+O\left[ \tau^{-1/28}\right] \right\}
\end{equation}
Inside the unit disk, the rest of our integration will be done over line segments.  We will choose these line segments in such a way that the value of the integral over them will be small.  We want to make sure that over the line segments the modulus of $\exp \mathcal{Q}(t)$ decreases, as $t$ moves away from $\alpha$, and is bounded above by the modulus of $\exp \mathcal{Q}(\alpha )$.  One way of doing this is by noticing that if $t\in \mathbb{D}$ has $\Re t<0$, then, the modulus of $\mathcal{E}_{d}(Tt)$ decreases, while that of $(Tt)^{i\tau}$ remains constant, as $T$ goes from $1$ to $1/|t|$.  In other words the modulus of $\mathcal{E}_{d}(Tt)$ decreases as as $T$ goes from $1$ to $1/|t|$.  A similar analysis shows that if $\Re t>0$, then, the modulus of $\mathcal{E}_{d}(Tt)$ increases as as $T$ goes from $0$ to $1$.  We can finally write down all of $\Delta_{d} \cap \mathbb{D}$.  Write $J$ (see (\ref{J}) )as $[a_{-},a_{+}]$.  Let $l_{-}$ be the line parametrized as $l_{-}(T)=T\Delta_{\alpha} (a_{-})$, where $T\in [1,|\Delta_{\alpha}(a_{-})|^{-1}]$.  Let $l_{+}$ be the line parametrized as $l_{+}(T)=T\Delta_{\alpha} (a_{+})$, where $T\in [-|\Delta_{\alpha}(a_{+})|^{-1},|\Delta_{\alpha}(a_{+})|^{-1}]$.  Then, our integration contour inside $\mathbb{D}$ will be:
\[
\Delta_{d}=l_{+}-\Delta_{\alpha}+l_{-}
\] 
One should note that $l_{+}$ actually touches the origin, which in principle is not allowed, since $\Gamma$ should be as in (\ref{integral}).  To see that this last does not pose a problem, just notice that integral of $\mathcal{E}_{d}(t)t^{x-1}$ on a straight line through the origin goes to 0 as the length of the line goes to 0.  To bound the values of the integrals over $l_{+}$ and $l_{-}$ we first note that:
\[
\int_{l_{\pm}}\mathcal{E}_{d}(t)t^{x_{d}-1}dt\leq (\max_{l_{\pm}}|\exp \mathcal{Q}|)(\min_{l_{\pm}}|1-t|)\int_{l_{\pm}}|t|^{\sigma_{d}-1} |dt|\ll \max_{l_{\pm}}|\exp \mathcal{Q}|
\]
Here we used the fact that the distance from $l_{\pm}$ to $1$ is bounded away from zero and that the integrals $\int_{l_{\pm}}|t|^{\sigma -1} |dt|$ are uniformly bounded above.  Now, by (\ref{asymptotic i}) and (\ref{lower bound F negative T}),
\[
\begin{split}
\exp [\mathcal{Q}(a_{\pm})]&=\exp [\mathcal{Q}(\alpha )]\exp\{ \phi [\Delta_{\alpha}(a_{\pm})]\} \\
&=\exp [\mathcal{Q}(\alpha )]\exp \left\{ -\frac{9}{4}\tau \frac{K_{2}^{3}}{K_{3}^{2}}F(a_{\pm})-G(a_{\pm})\right\} \\
&=\exp [\mathcal{Q}(\alpha )]\exp \left\{ -\frac{9}{4}\tau \frac{K_{2}^{3}}{K_{3}^{2}}F(a_{\pm})\right\} \{ 1+O(\tau K_{2}^{4}U(a_{\pm})^{4})\} \\
&=\exp [\mathcal{Q}(\alpha )]\exp \{ -D\tau^{1/7} \} \{ 1+O(\tau^{-1/28})\}, 
\end{split}
\]
where $D$ is some positive constant.  Then, 
\[
\int_{l_{\pm}}\mathcal{E}_{d}(t)t^{x-1}dt\ll |\Delta_{\alpha} (a_{\pm})|\exp [\mathcal{Q}(\alpha )]\exp \{ -D\tau^{1/7} \} \ll \mathcal{E}_{d}(\alpha )\alpha^{i\tau +1}\exp \{ -D\tau^{1/7} \} 
\]
We then have:
\begin{equation}
\begin{split}
\int_{\Delta_{d}}\mathcal{E}_{d}(t)t^{x-1}dt&=\left( \int_{-\Delta_{\alpha}}+\int_{l_{-}}+\int_{l_{+}}\right) \mathcal{E}_{d}(t)t^{x-1}dt\\
&=-i\frac{3}{2}\frac{K_{2}}{K_{3}}\mathcal{E}_{d}(\alpha )\alpha^{x}\mathcal{I}(\alpha )\{ 1+O[K_{2}^{-1}\exp (-D\tau^{1/7})]\}\\
&=-i\frac{3}{2}\frac{K_{2}}{K_{3}}\mathcal{E}_{d}(\alpha )\alpha^{x}\mathcal{I}(\alpha )\{ 1+O[\exp (-D\tau^{1/7})]\}
\end{split}
\end{equation}
In the last line above we used that $\sigma \in [\epsilon ,1-\epsilon ]$, and that $K_{2}=\sqrt{1-\tau^{2}/d^{2}}\gg d^{-5/12}$ since $0<\tau \leq d-\sqrt[6]{d}$.  Finally, we define the section of integration contour that lies outside $\mathbb{D}$.  For any given curve in $\gamma :[0,1]\rightarrow \mathbb{C}$ define $\gamma^{-1}(T)=[\gamma (T)]^{-1}$, $\overline{\gamma}(T)=\overline{\gamma (T)}$ and $-\gamma(T)=\gamma (1-T)$.  Finally, define $\Gamma_{d}\cap ( \mathbb{C}\setminus \mathbb{D}):=-\left( \overline{\Gamma_{d}\cap \mathbb{D}}\right)^{-1}$.  We then have:
\[
\begin{split}
\int_{-( \overline{\Gamma_{d}\cap \mathbb{D}})^{-1}}\mathcal{E}_{d}(t)t^{x_{d}-1}dt&=\int \mathcal{E}_{d}[-( \overline{\Gamma_{d}\cap \mathbb{D}})^{-1} ][-( \overline{\Gamma_{d}\cap \mathbb{D}})^{-1}]^{x_{d}-1}(-1)[-(\overline{\Gamma_{d}\cap \mathbb{D}})^{-2}][-(\overline{\Gamma_{d}\cap \mathbb{D}})]'\\
&=(-1)^{d}\int \overline{\mathcal{E}_{d}[-(\Gamma_{d}\cap \mathbb{D})][-(\Gamma_{d}\cap \mathbb{D})]^{-\overline{x}}[-( \Gamma_{d}\cap \mathbb{D})]'}\\
&=(-1)^{d+1}\overline{ \int \mathcal{E}_{d}[\Gamma_{d}\cap \mathbb{D}](\Gamma_{d}\cap \mathbb{D})^{-\overline{x}}( \Gamma_{d}\cap \mathbb{D})'}\\
&=(-1)^{d+1}\overline{\int_{\Gamma_{d}\cap \mathbb{D}}\mathcal{E}_{d}(t)t^{(1-\overline{x})-1}dt}
\end{split}
\] 
Note that $\Im{1-\overline{x}}=\tau$, so $\alpha$ and $\phi$ do not change.  We can then use all the above estimates to obtain, as $d\rightarrow +\infty$:
\begin{equation} \label{middle formula}
\int_{\Gamma_{d}}\mathcal{E}_{d}(t)t^{x-1}dt=\left[ \int_{\Gamma_{d}\cap \mathbb{D}}+\int_{\Gamma_{d}\cap ( \mathbb{C}\setminus \mathbb{D})}\right]\mathcal{E}_{d}(t)t^{x-1}dt=\mathcal{F}(d,x)+(-1)^{d+1}\overline{\mathcal{F}(d,1-\overline{x})},
\end{equation}
where,
\begin{equation}
\mathcal{F}(d,x)=-i\frac{3}{2}\frac{K_{2}}{K_{3}}\mathcal{E}_{d}(\alpha )\alpha^{x}\mathcal{I}_{d}(x)\{ 1+O[K_{2}^{-1}\exp (-D\tau^{1/7})]\} ,
\end{equation}
where,
\[
\mathcal{I}_{d}(x)=\int_{-1/3}^{\infty}\exp \left\{ -\frac{9}{4}\tau \frac{K_{2}^{3}}{K_{3}^{2}}F(T)\right\} U'(T)dT
\]
Note that $-i3K_{2}/2K_{3}=3C_{2}/2C_{3}$.  

\section{Some commentary}

The asymptotics of $\mathcal{L}_{d}(x)$ can be written down for all $x$ with $\Re{x}=-1/2$.  Remember that we worked only enough of the asymptotics of $\mathcal{L}_{d}$ to find its roots.  The methods used here seem to be perfectly applicable to the study of other polynomials and functions that have a nice enough generating series.  For example, if we take some sequence of polynomials $\{ \mathcal{P}_{d}\}_{d\geq 0}$, where $d$ is the degree of $\mathcal{P}_{d}$, such that the generating function $\sum_{m\geq 0}\mathcal{P}_{d}(m)t^{m}$ has all of its zeros on the unit circle, then, the methods used here can be applied under the appropriate, not very strict conditions.  These last topics will be subject of follow-up work.  Much of the work presented is inspired on the methods used in the theory of orthogonal polynomials and the theory of analytic combinatorics (for a very thorough introduction to the field of analytic combinatorics see Flajolet's and Sedgewick's book  \cite{citeulike:6390346}).  The polynomials $\mathcal{L}_{d}$ form a family of orthogonal polynomials which happen to be Mellin transforms of Laguerre polynomials (see \cite{bump-local}).  In general one can construct, through the use of recursions, families of orthogonal polynomials, which have generating series with nice properties, such as having generating functions with zeros on the unit circle.  The methods used in this paper can be applied to this last as well.  On the combinatorial geometrical side of things there might be interesting applications to knowledge about the roots of Ehrhart polynomials; see Henk's et al.\ work in \cite{2005math......7528H}.

\bibliographystyle{plain}  
\bibliography{polyref}        
\end{document}